\UseRawInputEncoding
\documentclass[11pt]{article}
\usepackage{amsmath}

\usepackage{amsmath,amssymb,amsthm,graphicx,mathrsfs}
\usepackage[numbers,sort&compress]{natbib}
\usepackage{color}
\usepackage[colorlinks,bookmarksopen,bookmarksnumbered,citecolor=black, linkcolor=blue, urlcolor=black]{hyperref}
\numberwithin{equation}{section}
\newcommand{\beq}{\begin{equation}}
\newcommand{\enq}{\end{equation}}

\newtheorem{Theorem}{Theorem}[section]
\newtheorem{Lemma}[Theorem]{Lemma}
\newtheorem{Corollary}[Theorem]{Corollary}
\newtheorem{Definition}[Theorem]{Definition}
\newtheorem{Remark}[Theorem]{Remark}

\newcommand{\benu}{\begin{enumerate}}
\newcommand{\beqa}{\begin{eqnarray}}
\newcommand{\beqan}{\begin{eqnarray*}}
\newcommand{\eay}{\end{array}}
\newcommand{\edm}{\end{displaymath}}
\newcommand{\eenu}{\end{enumerate}}
\newcommand{\eeq}{\end{equation}}
\newcommand{\eeqa}{\end{eqnarray}}
\newcommand{\eeqan}{\end{eqnarray*}}

\newcommand{\br}{\begin{Remark}}
\newcommand{\er}{\end{Remark}}

\newcommand{\bqa}{\begin{eqnarray}}
\newcommand{\eqa}{\end{eqnarray}}
\newcommand{\bqw}{\begin{eqnarray*}}
\newcommand{\eqw}{\end{eqnarray*}}

\newcommand{\bea}{\begin{array}{cc}}
\newcommand{\ena}{\end{array}}

\allowdisplaybreaks[4]
\begin{document}
\begin{center}
{\LARGE Pullback Attractors for a Non-Autonomous Plate System with Strong Damping and
Delay\\}
{\vspace{0.20in} Xuxiao Hou ${}^{\mathrm{1}}$ \quad Bin Yang ${}^{\mathrm{2},\dag}$$\ $ Yuming Qin $^{3}$$\ $ Alain Miranville $^{4}$\\}
\end{center}
{\small
$^{1}$ School of Automation and Electrical Engineering, Inner Mongolia University of Science and Technology, Baotou 014010, Inner Mongolia, China\\
$^{2}$ School of Science, Inner Mongolia University of Science and Technology, Baotou 014010, Inner Mongolia, China \\
$^{3}$ School of Mathematics and Statistics, Institute for Nonlinear Science, Donghua University, Shanghai 201620, China\\
$^{4}$ Laboratoire de Math\'ematiques Appliqu\'ees du Havre (LMAH), Universit\'e Le Havre Normandie, 76063 Le Havre Cedex, France
 \vspace{3mm}
}

\begin{abstract}
This paper addresses the existence of pullback attractors for a class of nonlinear non-autonomous strongly damped plate equations with delay. The model contains the biharmonic operator, the strong damping term, a nonlinear source term and a delay operator acting on the history of the solution. The interaction between the strong damping mechanism and the hereditary delay effect gives rise to several analytical difficulties in deriving uniform estimates and proving pullback asymptotic compactness. To overcome these difficulties, we construct a modified energy functional adapted to the strong damping structure and employ the contractive function method to handle the delay term in the history phase space. By establishing the existence and uniqueness of weak solutions, uniform pullback estimates and the pullback asymptotic compactness of the associated non-autonomous process, we prove the existence of pullback attractors for the strongly damped plate equation with delay.
\end{abstract}

\hspace{4mm}{\bf Keywords:} Pullback attractors; Non-autonomous plate equations; Strong damping; Delay

\hspace{4mm}{\bf 2020 MSC:} 35B41, 35B40, 35L75, 35R10

\section{Introduction}
\setcounter{equation}{0}
\let\thefootnote\relax\footnote{*Corresponding author: binyangdhu@163.com}
\let\thefootnote\relax\footnote{\footnotesize E-mails: hxximust2025@163.com, yuming@dhu.edu.cn, alain.miranville@univ-lehavre.fr}
\noindent
In this paper, we investigate the following non-autonomous plate equation with strong damping and delay:
\begin{equation}\label{1.1}
\left\{
\begin{alignedat}{3}
&u_{tt}-\Delta u_t+\Delta^2u+\lambda u+f(u)
=G(t,u^t)+h(x,t),
&\qquad& x\in\Omega,
&\qquad& t>\tau,\\
&u=\Delta u=0,
&\qquad& x\in\partial\Omega,
&\qquad& t>\tau,\\
&u(x,\tau+\theta)=\phi_0(x,\theta),
&\qquad& x\in\Omega,
&\qquad& \theta\in[-r,0],\\
&u_t(x,\tau+\theta)=\phi_1(x,\theta),
&\qquad& x\in\Omega,
&\qquad& \theta\in[-r,0].
\end{alignedat}
\right.
\end{equation}

Here \(\Omega\subset\mathbb R^N\) is a bounded domain with smooth boundary, \(N\geq1\), \(\lambda>0\), \(r>0\) is the length of the delay, \(h(x,t)\) is a time-dependent external force, \(f\) is the nonlinear term, and \(G(t,u^t)\) is a delay operator depending on the displacement history. For each \(t\geq\tau\), the history function \(u^t\) is defined by
\begin{equation}\label{1.2}
u^t(\theta)=u(t+\theta),\qquad \theta\in[-r,0].
\end{equation}

Throughout this paper, \(u^t\) denotes the history function, whereas \(u_t\) denotes the time derivative of \(u\). The term \(\Delta^2u\) shows that \eqref{1.1} is a plate-type equation with Navier boundary condition, while \(-\Delta u_t\) is the strong damping term. In the energy estimate, the strong damping term produces the dissipative quantity \((-\Delta u_t,u_t)=\|\nabla u_t\|^2\), which provides spatial dissipation for the velocity. Therefore, problem \eqref{1.1} combines the fourth-order structure of plate equations, the velocity-gradient dissipation of strongly damped systems and the hereditary feature of delayed equations.

Let \(H=L^2(\Omega)\) and \(V=H^2(\Omega)\cap H_0^1(\Omega)\). The norm and inner product in \(H\) are denoted by \(\|\cdot\|\) and \((\cdot,\cdot)\), respectively. For a Banach space \(X\), denote by \(C_X=C([-r,0];X)\) the space of continuous functions from \([-r,0]\) into \(X\), endowed with the norm \(\|\xi\|_{C_X}=\displaystyle\sup_{\theta\in[-r,0]}\|\xi(\theta)\|_X\). The ambient history space associated with \eqref{1.1} is
\begin{equation}\label{1.3}
\mathcal H=C_V\times C_H.
\end{equation}

Since both the displacement history and the velocity history are prescribed in \eqref{1.1}, not every element of \(\mathcal H\) is admissible as an initial datum. The two histories must be compatible with the relation \(u_t=\partial_tu\) on the initial interval. Therefore, we introduce the compatible initial history space
\begin{equation}\label{1.4}
\mathcal H_c=
\left\{
(\phi_0,\phi_1)\in C_V\times C_H:
\phi_0\in C^1([-r,0];H),\
\partial_\theta\phi_0=\phi_1\ \text{in } C_H
\right\}.
\end{equation}

The space \(\mathcal H_c\) is endowed with the norm inherited from \(\mathcal H\), namely \(\|(\phi_0,\phi_1)\|_{\mathcal H_c}:=\|(\phi_0,\phi_1)\|_{\mathcal H}\). The initial data of problem \eqref{1.1} are always assumed to satisfy \((\phi_0,\phi_1)\in\mathcal H_c\). Thus, the natural phase space for the process generated by \eqref{1.1} is not the usual product space \(V\times H\), but the compatible history space \(\mathcal H_c\).

We assume that
\begin{equation}\label{1.5}
f\in C^1(\mathbb R,\mathbb R),\qquad f(0)=0,
\end{equation}
and define
\begin{equation}\label{1.6}
F(s)=\int_0^s f(\xi)\,d\xi.
\end{equation}

To ensure that the energy functional is bounded from below, we suppose that there exists \(C_f>0\) such that
\begin{equation}\label{1.7}
F(s)\geq -C_f-\frac{\lambda}{4}s^2,\qquad s\in\mathbb R.
\end{equation}

Moreover, to derive dissipative estimates, we assume that there exist constants \(\mu_0>0\), \(\delta_0>0\) and \(C_0>0\), with \(\delta_0>0\) sufficiently small, such that
\begin{equation}\label{1.8}
f(s)s\geq \mu_0F(s)-\delta_0s^2-C_0,\qquad s\in\mathbb R.
\end{equation}

The growth condition on \(f\) is chosen according to the plate energy space \(V=H^2(\Omega)\cap H_0^1(\Omega)\). If \(N>4\), by the Sobolev embedding \(H^2(\Omega)\hookrightarrow L^{\frac{2N}{N-4}}(\Omega)\), we assume that there exists \(C>0\) such that
\begin{equation}\label{1.9}
|f(s)|\leq C(1+|s|^\rho),\qquad
1\leq \rho\leq \frac{N}{N-4}.
\end{equation}

The case \(\rho=\frac{N}{N-4}\) is the critical growth case with respect to \(H^2(\Omega)\), while \(1\leq \rho<\frac{N}{N-4}\) is the subcritical case. If \(1\leq N\leq4\), the embedding \(H^2(\Omega)\hookrightarrow L^q(\Omega)\) holds for every finite \(q\geq1\); hence finite polynomial growth is admissible.

For uniqueness and continuous dependence, we further assume that \(f\) is locally Lipschitz from bounded subsets of \(V\) into \(H\). More precisely, for every \(R>0\), there exists a constant \(L_R>0\) such that
\begin{equation}\label{1.10}
\|f(u)-f(v)\|\leq L_R\|u-v\|_V
\end{equation}
whenever \(\|u\|_V+\|v\|_V\leq R\).

Since \(V\hookrightarrow H\) continuously, every displacement history \(u^t\in C_V\) can also be regarded as an element of \(C_H\). Therefore, the delay operator \(G(t,u^t)\) is assumed to depend only on the displacement history in the \(C_H\)-topology. More precisely, although the full phase variable is the pair \((u^t,(u_t)^t)\in\mathcal H_c\), the delay term \(G(t,u^t)\) acts only on the first component \(u^t\), viewed as an element of \(C_H\).

For the delay term, we impose the following assumptions \(G:\mathbb R\times C_H\to H\) is well-defined and satisfies the following conditions: \({\rm (G1)}\) \(t\mapsto G(t,\xi)\) is measurable for every \(\xi\in C_H\); \({\rm (G2)}\) \(G(t,0)=0\) for all \(t\in\mathbb R\); \({\rm (G3)}\) there exists a constant \(L_G>0\) such that \(\|G(t,\xi)-G(t,\eta)\|\leq L_G\|\xi-\eta\|_{C_H}\) for all \(t\in\mathbb R\) and \(\xi,\eta\in C_H\); \({\rm (G4)}\) there exist constants \(m_0>0\) and \(C_1>0\) such that, for every \(m\in[0,m_0]\), \(\int_{\tau}^{t}e^{ms}\|G(s,u^s)-G(s,v^s)\|^2\,ds\leq C_1\int_{\tau-r}^{t}e^{ms}\|u(s)-v(s)\|^2\,ds\).

The external force is assumed to satisfy
\begin{equation}\label{1.11}
h\in L^2_{\mathrm{loc}}(\mathbb R;H).
\end{equation}

For the pullback dynamics, we assume that the effective dissipation dominates the possible growth caused by the delay term. More precisely, let $\beta>0$ be the dissipation rate obtained from the modified energy estimate, and let $\delta_G>0$ be the constant generated by the delay estimate through condition (G4). We set
\begin{equation}\label{1.12}
\alpha=\beta-\delta_G>0.
\end{equation}

Moreover, the external force is assumed to satisfy
\begin{equation}\label{1.13}
\int_{-\infty}^{t} e^{\alpha s}\|h(s)\|^{2}\,ds<+\infty,\quad t\in \mathbb R.
\end{equation}

Plate equations constitute an important class of fourth-order evolution equations describing the vibration and deformation of elastic structures. Owing to the presence of the biharmonic operator, their long-time dynamics are generally studied in higher-order energy spaces and require compactness arguments different from those used for second-order wave equations. The abstract theory of pullback attractors and non-autonomous dynamical systems has been systematically developed in \cite{Bortolan2020,Carvalho2013,Chepyzhov2002,Temam1997} and subsequently applied to various plate-type systems. Yang and Zhong \cite{Yang2008} studied a plate equation with nonlinear damping and proved the existence of a global attractor. Closely related to the present work, Carbone et al.\ \cite{Carbone2011} considered the singularly non-autonomous plate equation
\(u_{tt}+a(t,x)u_t-\Delta u_t+\Delta^2u+\lambda u=f(u)\) in
\(\bigl(H^2(\Omega)\cap H_0^1(\Omega)\bigr)\times L^2(\Omega)\), subject to the Navier boundary conditions
\(u=\Delta u=0\) on \(\partial\Omega\), where \(a(t,x)\) denotes a damping coefficient and \(\lambda>0\). Thus, their model contains the same biharmonic operator and strong damping mechanism as the present equation, but its phase variable records only \((u,u_t)\) and does not include a general delay operator acting on the displacement history. This distinction is essential because hereditary terms require estimates over an interval of past states rather than only at the current time. Related results for strongly damped, thermoelastic and other non-autonomous plate systems can be found in \cite{Aouadi2023,Aouadi2024,Bezerra2026}.

Delay effects have also been incorporated into several plate models. Aouadi \cite{Aouadi2020} studied an extensible thermoelastic plate equation with time-varying delayed feedback and established the existence of global and exponential attractors. The system consists of the plate equation \(u_{tt}+\Delta^2u-M\bigl(\|\nabla u\|^2\bigr)\Delta u+\mu_1u_t+\mu_2u_t\bigl(t-\tau(t)\bigr)+\vartheta=h(x,t)\), coupled with the heat equation \(\vartheta_t-\Delta\vartheta-\Delta u_t=0\). Here, \(M\bigl(\|\nabla u\|^2\bigr)\) is a Kirchhoff-type nonlocal coefficient, \(\vartheta\) denotes the temperature variable, \(\mu_1u_t\) is the instantaneous damping term, and \(\mu_2u_t\bigl(t-\tau(t)\bigr)\) is the delayed velocity feedback. Consequently, the phase space contains the variables \((u,u_t,\vartheta)\) together with an associated delay-history component. Variable-delay, Kirchhoff and stochastic delayed plate equations were further studied in \cite{Qin2023,Yao2023}. However, these works mainly involve discrete delay terms such as \(u_t\bigl(t-\tau(t)\bigr)\), whereas the operator \(G(t,u^t)\) considered here depends on the entire displacement history \(u^t(\theta)=u(t+\theta)\) for \(\theta\in[-r,0]\). Since \(u_t\) denotes the current velocity, whereas \(u^t\) denotes the displacement history, they represent fundamentally different quantities. To record both the displacement and velocity histories,and formulate the associated process in the compatible history space \(\mathcal H_c=\left\{(\phi_0,\phi_1)\in C_V\times C_H:\phi_0\in C^1([-r,0];H),\ \partial_\theta\phi_0=\phi_1\ \text{in }C_H\right\}\). The compatibility condition \(\partial_\theta\phi_0=\phi_1\) guarantees that the prescribed displacement and velocity histories originate from the same solution trajectory. Therefore, the simultaneous treatment of the structural strong damping term \(-\Delta u_t\), the hereditary delay operator \(G(t,u^t)\), and the compatibility between the two history components constitutes the principal distinction of the present problem.

The main difficulties and contributions of this paper are summarized as follows.

(i) The simultaneous presence of the biharmonic operator \(\Delta^2u\), the strong damping term \(-\Delta u_t\), the nonlinear source \(f(u)\), the hereditary delay \(G(t,u^t)\), and the non-autonomous force \(h(x,t)\) produces a complicated coupled energy structure. To close the dissipative estimate, we introduce the modified energy functional \(L(t)=2E(t)+\varepsilon (2(u_t(t),u(t))+\|\nabla u(t)\|^2)+K_0\) and combine it with the weighted delay estimate , where E(t) is the standard energy functional defined in (3.29). Under \(\beta>\delta_G\), this yields the decay exponent \(\alpha=\beta-\delta_G>0\) and a pullback \(\mathcal D_\alpha\)-absorbing family.

(ii) To capture the intrinsic kinematic relation between the displacement and velocity histories, we do not treat them as two independent components of \(C_V\times C_H\). Instead, we introduce the compatible history space \(\mathcal H_c\), where the compatibility condition ensures that both histories originate from the same solution trajectory. We prove that this compatibility is preserved by the evolution and establish the existence, uniqueness and continuous dependence of weak solutions. Consequently, \(U(t,\tau)(\phi_0,\phi_1)=(u^t,(u_t)^t)\) defines a well-posed continuous process on \(\mathcal H_c\), providing the rigorous phase-space framework required for the subsequent pullback analysis.

(iii) The principal compactness difficulty is that the Sobolev embedding at the critical growth exponent is continuous but no longer compact, while the hereditary delay requires uniform control of the entire history segment. To overcome this difficulty, we combine the additional regularity generated by the strong damping term \(-\Delta u_t\) with the contractive-function estimate \(\|U(t,t-T)\xi_1-U(t,t-T)\xi_2\|_{\mathcal H_c}^2\leq\varepsilon+\Psi_{t,T}(\xi_1,\xi_2)\). By proving that \(\Psi_{t,T}\) has the required double-limit vanishing property on the pullback absorbing family, we recover compactness of both the displacement and velocity histories without relying on the compactness of the critical Sobolev embedding. This establishes pullback \(\mathcal D_\alpha\)-asymptotic compactness in \(\mathcal H_c\), including the critical-growth case, and consequently yields a unique pullback \(\mathcal D_\alpha\)-attractor.

The paper is organized as follows. In Section 2, we recall the basic notions of compatible phase spaces, pullback absorbing families and pullback attractors. In Section 3, we prove the existence, uniqueness and continuous dependence of weak solutions with initial histories in \(\mathcal H_c\). In Section 4, we establish the pullback absorbing family and the pullback asymptotic compactness of the process in \(\mathcal H_c\), and then obtain the existence and uniqueness of the pullback \(\mathcal D_\alpha\)-attractor in \(\mathcal H_c\).

\section{Preliminaries}

In this section, we recall some basic notions and abstract results concerning pullback attractors, which will be used to study the long-time behavior of problem \eqref{1.1}.

For two nonempty subsets \(B_1,B_2\subset\mathcal H_c\), the Hausdorff semi-distance from \(B_1\) to \(B_2\) is defined by
\[
\operatorname{dist}_{\mathcal H_c}(B_1,B_2)
=
\sup_{\xi\in B_1}\inf_{\eta\in B_2}
\|\xi-\eta\|_{\mathcal H_c}.
\]
We denote by \(\mathcal P(\mathcal H_c)\) the family of all nonempty subsets of \(\mathcal H_c\).

A family \(\{U(t,\tau)\}_{t\geq\tau}\) is called a process on \(\mathcal H_c\) if \(U(t,\tau):\mathcal H_c\to\mathcal H_c\) is continuous for every \(t\geq\tau\), \(U(\tau,\tau)=I\), and
\[
U(t,\tau)=U(t,s)U(s,\tau),
\qquad
\tau\leq s\leq t.
\]

Let \(\mathcal D\) be a nonempty class of parameterized families \(\widehat D=\{D(t):t\in\mathbb R\}\subset\mathcal P(\mathcal H_c)\).

\begin{Definition}\label{def2.1}
The process \(\{U(t,\tau)\}_{t\geq\tau}\) is said to be pullback \(\mathcal D\)-asymptotically compact in \(\mathcal H_c\) if, for any \(t\in\mathbb R\), any \(\widehat D\in\mathcal D\), any sequence \(\tau_n\to-\infty\), and any sequence \(\xi_n\in D(\tau_n)\), the sequence
\[
\{U(t,\tau_n)\xi_n\}_{n=1}^{\infty}
\]
is precompact in \(\mathcal H_c\).
\end{Definition}

\begin{Definition}\label{def2.2}
A family \(\widehat B=\{B(t):t\in\mathbb R\}\subset\mathcal P(\mathcal H_c)\) is called a pullback \(\mathcal D\)-absorbing family for the process \(\{U(t,\tau)\}_{t\geq\tau}\) if, for any \(t\in\mathbb R\) and any \(\widehat D\in\mathcal D\), there exists \(\tau_0=\tau_0(t,\widehat D)\leq t\) such that
\[
U(t,\tau)D(\tau)\subset B(t),
\qquad
\tau\leq\tau_0.
\]
\end{Definition}

\begin{Definition}\label{def2.3}
A family \(\widehat{\mathcal A}=\{\mathcal A(t):t\in\mathbb R\}\subset\mathcal P(\mathcal H_c)\) is called a pullback \(\mathcal D\)-attractor for the process \(\{U(t,\tau)\}_{t\geq\tau}\) in \(\mathcal H_c\) if the following conditions hold:

\par\noindent
(i) \(\mathcal A(t)\) is compact in \(\mathcal H_c\) for every \(t\in\mathbb R\);

\par\noindent
(ii) \(\widehat{\mathcal A}\) pullback attracts every \(\widehat D\in\mathcal D\), that is,
\[
\lim_{\tau\to-\infty}
\operatorname{dist}_{\mathcal H_c}
\bigl(U(t,\tau)D(\tau),\mathcal A(t)\bigr)=0,
\qquad
t\in\mathbb R;
\]

\par\noindent
(iii) \(\widehat{\mathcal A}\) is invariant, namely,
\[
U(t,\tau)\mathcal A(\tau)=\mathcal A(t),
\qquad
-\infty<\tau\leq t<+\infty.
\]

\end{Definition}

For the pullback dynamics of \eqref{1.1}, we use the following tempered class of parameterized families. For a given \(\alpha>0\), denote by \(\mathcal D_\alpha\) the class of all families \(\widehat D=\{D(t):t\in\mathbb R\}\subset\mathcal P(\mathcal H_c)\) satisfying
\[
\lim_{\tau\to-\infty}
e^{\alpha\tau}
\sup_{\xi\in D(\tau)}
\|\xi\|_{\mathcal H_c}^{2}=0.
\]
This class is consistent with the exponential pullback estimate obtained later from the dissipative structure of \eqref{1.1}.

\begin{Definition}\label{def2.4}
Let \(B\) be a bounded subset of \(\mathcal H_c\). A nonnegative function \(\Psi(\cdot,\cdot)\) defined on \(B\times B\) is called a contractive function on \(B\times B\) if, for any sequence \(\{\xi_n\}_{n=1}^{\infty}\subset B\), there exists a subsequence \(\{\xi_{n_k}\}_{k=1}^{\infty}\) such that
\[
\lim_{k\to\infty}\lim_{l\to\infty}
\Psi(\xi_{n_k},\xi_{n_l})=0.
\]
The set of all contractive functions on \(B\times B\) is denoted by \(\operatorname{Contr}(B)\).
\end{Definition}

\begin{Theorem}\label{thm2.5}
Let \(\{U(t,\tau)\}_{t\geq\tau}\) be a process on \(\mathcal H_c\), and suppose that it has a pullback \(\mathcal D\)-absorbing family \(\widehat B=\{B(t):t\in\mathbb R\}\). Assume that, for any \(t\in\mathbb R\) and any \(\varepsilon>0\), there exist \(T=T(t,\widehat B,\varepsilon)>0\) and \(\Psi_{t,T}\in\operatorname{Contr}(B(t-T))\) such that
\[
\|U(t,t-T)\xi_1-U(t,t-T)\xi_2\|_{\mathcal H_c}^{2}
\leq
\varepsilon+\Psi_{t,T}(\xi_1,\xi_2)
\]
for all \(\xi_1,\xi_2\in B(t-T)\). Then the process \(\{U(t,\tau)\}_{t\geq\tau}\) is pullback \(\mathcal D\)-asymptotically compact in \(\mathcal H_c\).
\end{Theorem}

\begin{Theorem}\label{thm2.6}
Let \(\{U(t,\tau)\}_{t\geq\tau}\) be a continuous process on \(\mathcal H_c\). If the following two conditions hold:
\par\noindent
(i) \(\{U(t,\tau)\}_{t\geq\tau}\) has a pullback \(\mathcal D\)-absorbing family in \(\mathcal H_c\);

\par\noindent
(ii) \(\{U(t,\tau)\}_{t\geq\tau}\) is pullback \(\mathcal D\)-asymptotically compact in \(\mathcal H_c\),

\par\noindent
then the process possesses a unique pullback \(\mathcal D\)-attractor \(\widehat{\mathcal A}=\{\mathcal A(t):t\in\mathbb R\}\) in \(\mathcal H_c\).

\end{Theorem}

\section{Existence and Uniqueness of Solutions}
\setcounter{equation}{0}

In this section, we prove the existence, uniqueness and continuity of weak solutions to equation \eqref{1.1}.
\begin{Lemma}\label{lem4.1}
A function \(u\) is called a weak solution of equation \eqref{1.1} on
\([\tau-r,T]\) with initial history
\((\phi_0,\phi_1)\in\mathcal H_c\) if
\[
u(t)=\phi_0(t-\tau),\qquad
u_t(t)=\phi_1(t-\tau),\qquad t\in[\tau-r,\tau],
\]
and
\[
u\in L^\infty(\tau,T;V),\qquad
u_t\in L^\infty(\tau,T;H)\cap L^2(\tau,T;H_0^1(\Omega)),
\qquad
u_{tt}\in L^2(\tau,T;V'),
\]
such that, for every \(\varphi\in V\) and almost every
\(t\in(\tau,T)\),
\begin{equation}\label{3.1}
\begin{aligned}
&( u_{tt}(t),\varphi)
+(\nabla u_t(t),\nabla\varphi)
+(\Delta u(t),\Delta\varphi)
+\lambda(u(t),\varphi)
+(f(u(t)),\varphi)  \\
&\qquad
=(G(t,u^t),\varphi)+(h(t),\varphi).
\end{aligned}
\end{equation}
\end{Lemma}
The following theorem gives the existence of weak solutions by the Faedo--Galerkin method.

\begin{Theorem}\label{thm3.1}
Let \(\tau\in\mathbb R\), \(T>\tau\), and let \((\phi_0,\phi_1)\in\mathcal H_c\). Then equation \eqref{1.1} possesses a weak solution \(u=u(\cdot;\tau,\phi_0,\phi_1)\) on \([\tau-r,T]\). Moreover,
\begin{equation}\label{3.2}
u\in C([\tau-r,T];V),
\qquad
u_t\in C([\tau-r,T];H),
\end{equation}
and, for every \(t\in[\tau,T]\), the history pair \(\bigl(u^t,(u_t)^t\bigr)\) belongs to \(\mathcal H_c\).
\end{Theorem}

{\bf Proof.}
Let \(\{w_j\}_{j\geq1}\) be an orthonormal basis of \(H\)
consisting of eigenfunctions of the Dirichlet Laplacian,
with \(w_j\in V\) for every \(j\geq1\). Thus \(-\Delta w_j=\mu_jw_j\), \(w_j|_{\partial\Omega}=0\), \(0<\mu_1\leq \mu_2\leq\cdots\), and \({\rm span}\{w_j:j\geq1\}\) is dense in \(V\). For each \(m\in\mathbb N\), set \(V_m={\rm span}\{w_1,\ldots,w_m\}\), and let \(P_m\) be the \(H\)-orthogonal projection from \(H\) onto \(V_m\). Since the basis is generated by the Dirichlet Laplacian, \(P_m\) is uniformly bounded on \(V\), namely, \(\|P_mz\|_V\leq C\|z\|_V\) for \(z\in V\).

We seek an approximate solution of the form \(u_m(t)=\sum\limits_{j=1}^{m}d_{jm}(t)w_j\). The Galerkin system is given by
\begin{equation}\label{3.3}
\left\{
\begin{alignedat}{2}
&(u_m''(t),w_j)+(\nabla u_m'(t),\nabla w_j)
 +(\Delta u_m(t),\Delta w_j)
 +\lambda(u_m(t),w_j)
 +(f(u_m(t)),w_j) &&\\
&\qquad =(G(t,u_m^t),w_j)+(h(t),w_j),
 \qquad \quad\quad\quad\quad\quad\quad \text{for a.e. } t>\tau,
 \quad 1\leq j\leq m,\\
&u_m(t)=P_m\phi_0(t-\tau),
 \qquad u_m'(t)=P_m\phi_1(t-\tau),
 \qquad t\in[\tau-r,\tau].
\end{alignedat}
\right.
\end{equation}

Since \((\phi_0,\phi_1)\in\mathcal H_c\), we have \(\phi_0\in C([-r,0];V)\cap C^1([-r,0];H)\) and \(\partial_\theta\phi_0=\phi_1\) in \(C_H\). Hence the prescribed histories in \eqref{3.3} are compatible with the relation \(u_m'=\partial_tu_m\) on \([\tau-r,\tau]\). By the continuity and local Lipschitz property of the right-hand side with respect to the finite-dimensional coefficients, it admits a local solution. We shall derive estimates independent of \(m\), which imply that this solution can be extended to the whole interval \([\tau,T]\).

Multiplying the first equation in \eqref{3.3} by \(d'_{jm}(t)\), summing over \(j=1,\ldots,m\), we obtain
\begin{equation}\label{3.4}
\frac{d}{dt}E_m(t)+\|\nabla u_m'(t)\|^2
=(G(t,u_m^t),u_m'(t))+(h(t),u_m'(t)),
\end{equation}
where
\begin{equation}\label{3.5}
E_m(t)=
\frac12\|u_m'(t)\|^2
+\frac12\|\Delta u_m(t)\|^2
+\frac{\lambda}{2}\|u_m(t)\|^2
+\int_\Omega F(u_m(t))\,dx.
\end{equation}

Applying the Cauchy--Schwarz, Poincar\'e, and Young inequalities, we derive
\begin{equation}\label{3.6}
\begin{aligned}
|(G(t,u_m^t),u_m'(t))|
&\leq \|G(t,u_m^t)\|\,\|u_m'(t)\|  \\
&\leq C\|G(t,u_m^t)\|\,\|\nabla u_m'(t)\|  \\
&\leq \frac14\|\nabla u_m'(t)\|^2+C\|G(t,u_m^t)\|^2,
\end{aligned}
\end{equation}
and similarly
\begin{equation}\label{3.7}
|(h(t),u_m'(t))|
\leq
\frac14\|\nabla u_m'(t)\|^2+C\|h(t)\|^2.
\end{equation}

Combining \eqref{3.4}, \eqref{3.6} and \eqref{3.7}, we get
\begin{equation}\label{3.8}
\frac{d}{dt}E_m(t)+\frac12\|\nabla u_m'(t)\|^2
\leq
C\|G(t,u_m^t)\|^2+C\|h(t)\|^2.
\end{equation}

By \eqref{1.7}, there exist constants \(C_2, C_3>0\), independent of \(m\), such that
\begin{equation}\label{3.9}
E_m(t)+C_2
\geq
C_3\left(
\|u_m'(t)\|^2+\|\Delta u_m(t)\|^2+\|u_m(t)\|^2
\right).
\end{equation}

Since \(G(t,0)=0\), it follows from {\rm (G3)} that \(\|G(t,u_m^t)\|^2\leq C\|u_m^t\|_{C_H}^2\). Let \(M_m(t)=\sup\limits_{\tau\leq s\leq t}\bigl(E_m(s)+C_0\bigr)\). For \(\tau\leq s\leq t\), using the initial history on \([\tau-r,\tau]\) and \eqref{3.9}, we conclude
\begin{equation}\label{3.10}
\|u_m^s\|_{C_H}^2
\leq
C\|\phi_0\|_{C_H}^2+C M_m(s).
\end{equation}

Integrating \eqref{3.8} over \([\tau,t]\), using \eqref{3.9} and \eqref{3.10}, we infer that
\begin{equation}\label{3.11}
M_m(t)
\leq
C\left(
1+\|\phi_0\|_{C_V}^2+\|\phi_1\|_{C_H}^2
+\int_{\tau}^{T}\|h(s)\|^2\,ds
\right)
+C\int_{\tau}^{t}M_m(s)\,ds.
\end{equation}

Therefore, by the Gronwall lemma, we obtain
\begin{equation}\label{3.12}
M_m(t)
\leq
C_4\left(
1+\|\phi_0\|_{C_V}^2+\|\phi_1\|_{C_H}^2
+\int_{\tau}^{T}\|h(s)\|^2\,ds
\right),
\qquad t\in[\tau,T],
\end{equation}
where \(C_4>0\) is independent of \(m\). Consequently, \(\{u_m\}\) is bounded in \(L^\infty(\tau,T;V)\), \(\{u_m'\}\) is bounded in \(L^\infty(\tau,T;H)\), and \(\{u_m'\}\) is bounded in \(L^2(\tau,T;H_0^1(\Omega))\).

Before estimating the second time derivative, we record the bound for the nonlinear term. By \eqref{1.9}, we conclude
\begin{equation}\label{3.13}
|f(u_m)|^2
\leq
C\left(1+|u_m|^{2\rho}\right).
\end{equation}

When \(N>4\), the Sobolev embedding \(H^2(\Omega)\hookrightarrow L^{\frac{2N}{N-4}}(\Omega)\), together with \(1\leq \rho\leq \frac{N}{N-4}\), gives \(2\rho\leq \frac{2N}{N-4}\). Hence, from the uniform boundedness of \(\{u_m\}\) in \(L^\infty(\tau,T;V)\), it follows that
\begin{equation}\label{3.14}
\begin{aligned}
\int_{\tau}^{T}\|f(u_m(s))\|^2\,ds
&\leq
C\int_{\tau}^{T}\left(1+\|u_m(s)\|_{L^{2\rho}(\Omega)}^{2\rho}\right)\,ds  \\
&\leq
C\int_{\tau}^{T}\left(1+\|u_m(s)\|_{V}^{2\rho}\right)\,ds
\leq C_T .
\end{aligned}
\end{equation}

When \(1\leq N\leq4\), the same conclusion follows from the embedding \(H^2(\Omega)\hookrightarrow L^q(\Omega)\) for every finite \(q\geq1\) and the finite polynomial growth assumption on \(f\). Thus, in all admissible cases,
\begin{equation}\label{3.15}
\{f(u_m)\}\quad \text{is bounded in } L^2(\tau,T;H).
\end{equation}

Next, we estimate the second time derivative. From \eqref{3.3}, in \(V'\), we conclude
\begin{equation}\label{3.16}
u_m''
=
\Delta u_m'
-\Delta^2u_m
-\lambda u_m
-f(u_m)
+P_mG(t,u_m^t)
+P_mh(t).
\end{equation}

For any \(\varphi\in V\) with \(\|\varphi\|_V\leq1\), we derive
\begin{equation}\label{3.17}
\begin{aligned}
|\langle u_m'',\varphi\varphi\rangle|
&\leq
|(\nabla u_m',\nabla\varphi)|
+|(\Delta u_m,\Delta\varphi)|
+\lambda |(u_m,\varphi)|  \\
&+|(f(u_m),\varphi)|
+|(G(t,u_m^t),\varphi)|
+|(h(t),\varphi)|  \\
&\leq
C\|\nabla u_m'\|
+C\|\Delta u_m\|
+C\|u_m\|
+C\|f(u_m)\|\\
&+C\|G(t,u_m^t)\|
+C\|h(t)\|.
\end{aligned}
\end{equation}

Thus, by \eqref{3.12}, \eqref{3.15}, \eqref{3.17} and {\rm (G3)}, we obtain
\begin{equation}\label{3.18}
\{u_m''\}\quad \text{is bounded in } L^2(\tau,T;V').
\end{equation}

It follows from \eqref{3.12} and \eqref{3.18} that there exist a subsequence, still denoted by \(\{u_m\}\), and a function \(u\) such that
\begin{equation}\label{3.19}
\begin{alignedat}{2}
u_m
&\rightharpoonup^\ast u
&\qquad& \text{in } L^\infty(\tau,T;V),\\
u_m'
&\rightharpoonup^\ast u_t
&\qquad& \text{in } L^\infty(\tau,T;H),\\
u_m'
&\rightharpoonup u_t
&\qquad& \text{in } L^2(\tau,T;H_0^1(\Omega)),\\
u_m''
&\rightharpoonup u_{tt}
&\qquad& \text{in } L^2(\tau,T;V').
\end{alignedat}
\end{equation}

Since \(V\hookrightarrow\hookrightarrow H\hookrightarrow V'\), the Aubin--Lions compactness theorem yields, up to a subsequence,
\begin{equation}\label{3.20}
u_m\to u \quad \text{strongly in } L^2(\tau,T;H),
\qquad
u_m(x,t)\to u(x,t)\quad \text{a.e. in } \Omega\times(\tau,T).
\end{equation}

Moreover, \(P_m\phi_0\to\phi_0\) in \(C_H\). Hence, using {\rm (G4)}, we obtain
\begin{equation}\label{3.21}
\int_{\tau}^{T}\|G(s,u_m^s)-G(s,u^s)\|^2\,ds
\leq
C_G^2\int_{\tau-r}^{T}\|u_m(s)-u(s)\|^2\,ds
\to0.
\end{equation}
Thus, \(G(t,u_m^t)\to G(t,u^t)\) strongly in \(L^2(\tau,T;H)\).

We now identify the weak limit of the nonlinear term. Since \(u_m(x,t)\to u(x,t)\) a.e. in \(\Omega\times(\tau,T)\) and \(f\in C^1(\mathbb R,\mathbb R)\), we derive
\begin{equation}\label{3.22}
f(u_m(x,t))\to f(u(x,t))
\quad \text{a.e. in } \Omega\times(\tau,T).
\end{equation}

By \eqref{3.15}, there exist a subsequence, still denoted by \(\{f(u_m)\}\), and \(\chi\in L^2(\tau,T;H)\) such that
\begin{equation}\label{3.23}
f(u_m)\rightharpoonup \chi
\quad \text{weakly in } L^2(\tau,T;H).
\end{equation}

Combining \eqref{3.22} with  \eqref{3.23}, the standard weak convergence identification theorem implies \(\chi=f(u)\). Consequently,
\begin{equation}\label{3.24}
f(u_m)\rightharpoonup f(u)
\quad \text{weakly in } L^2(\tau,T;H).
\end{equation}

Passing to the limit in \eqref{3.3}, using \eqref{3.19}--\eqref{3.21} and \eqref{3.24}, we obtain \eqref{3.1}. The initial history follows from the convergence of \(P_m\phi_0\) and \(P_m\phi_1\). Therefore, \(u\) is a weak solution of equation \eqref{1.1}.

It remains to prove the strong continuity of the weak solution. From \eqref{3.19} and the standard Lions--Magenes theorem, we conclude
\begin{equation}\label{3.25}
u\in C_w([\tau,T];V),
\qquad
u_t\in C_w([\tau,T];H),
\end{equation}
where \(C_w\) denotes weak continuity. In addition, since \(u\in L^2(\tau,T;V)\) and \(u_t\in L^2(\tau,T;H)\), it follows that
\begin{equation}\label{3.26}
u\in C([\tau,T];H).
\end{equation}

We next show that the energy identity holds for the weak solution. For the Galerkin approximations, the identity corresponding to \eqref{3.4} reads
\begin{equation}\label{3.27}
E_m(t)+\int_s^t\|\nabla u_m'(r)\|^2\,dr
=
E_m(s)+\int_s^t (G(r,u_m^r)+h(r),u_m'(r))\,dr
\end{equation}
for all \(\tau\leq s\leq t\leq T\). Passing to the limit in \eqref{3.27}, using \eqref{3.19}--\eqref{3.24}, the strong convergence of \(G(r,u_m^r)\) and the weak convergence of \(u_m'\), we derive
\begin{equation}\label{3.28}
E(t)+\int_s^t\|\nabla u_t(r)\|^2\,dr
=
E(s)+\int_s^t (G(r,u^r)+h(r),u_t(r))\,dr
\end{equation}
 for all \(\tau\leq s\leq t\leq T\), where
\begin{equation}\label{3.29}
E(t)=
\frac12\|u_t(t)\|^2
+\frac12\|\Delta u(t)\|^2
+\frac{\lambda}{2}\|u(t)\|^2
+\int_\Omega F(u(t))\,dx.
\end{equation}

We now prove that \(t\mapsto\int_\Omega F(u(t))\,dx\) is continuous on \([\tau,T]\). Let \(t_n\to t\). By \eqref{3.25}, it follows that \(u(t_n)\rightharpoonup u(t)\) weakly in \(V\). Since \(u\in C([\tau,T];H)\), we also have \(u(t_n)\to u(t)\) strongly in \(H\). By interpolation and compact Sobolev embeddings, it follows that
\begin{equation}\label{3.30}
u(t_n)\to u(t)
\quad \text{strongly in } L^{\rho+1}(\Omega)
\end{equation}
when \(N>4\), where \(1\leq \rho\leq \frac{N}{N-4}\). For \(1\leq N\leq4\), the same conclusion holds for every finite polynomial exponent involved in the growth of \(F\). In view of \eqref{3.30}, the boundedness of \(\{u(t_n)\}\) in \(V\), and the growth estimate \(|F(s)|\leq C(1+|s|^{\rho+1})\), the Vitali convergence theorem yields
\begin{equation}\label{3.31}
\int_\Omega F(u(t_n))\,dx
\to
\int_\Omega F(u(t))\,dx.
\end{equation}

Therefore, we conclude
\begin{equation}\label{3.32}
t\mapsto\int_\Omega F(u(t))\,dx
\quad \text{is continuous on } [\tau,T].
\end{equation}

Since \(G(\cdot,u)\in L^2(\tau,T;H)\), \(h\in L^2(\tau,T;H)\), and \(u_t\in L^2(\tau,T;H_0^1(\Omega))\subset L^2(\tau,T;H)\), the right-hand side of \eqref{3.28} is absolutely continuous with respect to \(t\). Hence, \(E(t)\) is continuous on \([\tau,T]\). Combining the continuity of \(E(t)\), \eqref{3.26} and \eqref{3.32}, we derive
\begin{equation}\label{3.33}
t\mapsto
\left(
\|u_t(t)\|^2+\|\Delta u(t)\|^2
\right)
\quad \text{is continuous on } [\tau,T].
\end{equation}

We finally prove the strong continuity in \(V\times H\). For any sequence \(t_n\to t\), it follows from \eqref{3.25} that \(u(t_n)\rightharpoonup u(t)\) weakly in \(V\) and \(u_t(t_n)\rightharpoonup u_t(t)\) weakly in \(H\). Therefore, we obtain
\begin{equation}\label{3.34}
\|\Delta u(t)\|^2\leq \liminf_{n\to\infty}\|\Delta u(t_n)\|^2,
\qquad
\|u_t(t)\|^2\leq \liminf_{n\to\infty}\|u_t(t_n)\|^2.
\end{equation}

On the other hand, \eqref{3.33} gives
\begin{equation}\label{3.35}
\lim_{n\to\infty}
\left(
\|u_t(t_n)\|^2+\|\Delta u(t_n)\|^2
\right)
=
\|u_t(t)\|^2+\|\Delta u(t)\|^2.
\end{equation}

Hence both norms converge separately, that is,
\begin{equation}\label{3.36}
\|u_t(t_n)\|\to\|u_t(t)\|,
\qquad
\|\Delta u(t_n)\|\to\|\Delta u(t)\|.
\end{equation}

Weak convergence together with convergence of norms implies strong convergence in Hilbert spaces. Thus,
\begin{equation}\label{3.37}
u_t(t_n)\to u_t(t)\quad \text{strongly in }H,
\qquad
\Delta u(t_n)\to \Delta u(t)\quad \text{strongly in }H.
\end{equation}

In view of the equivalence between the norm \(\|\Delta\cdot\|_{L^2(\Omega)}\) and the standard norm on \(V\) under the boundary conditions \(u=\Delta u=0\) on \(\partial\Omega\), we deduce
\begin{equation}\label{3.38}
u\in C([\tau,T];V),
\qquad
u_t\in C([\tau,T];H).
\end{equation}

Together with the prescribed histories on \([\tau-r,\tau]\), \eqref{3.38} gives
\begin{equation}\label{3.39}
u\in C([\tau-r,T];V),
\qquad
u_t\in C([\tau-r,T];H).
\end{equation}

For every \(t\in[\tau,T]\), the history \(u^t\) belongs to \(C_V\), and \((u_t)^t\) belongs to \(C_H\). Moreover, for \(\theta\in[-r,0]\), one has \(\partial_\theta u^t(\theta)=u_t(t+\theta)=(u_t)^t(\theta)\) in \(H\). Consequently, we obtain
\begin{equation}\label{3.40}
\bigl(u^t,(u_t)^t\bigr)\in\mathcal H_c,
\qquad t\in[\tau,T].
\end{equation}
\(\hfill\Box\)

We next prove that the weak solution of equation \eqref{1.1} is unique and depends continuously on the initial history \((\phi_0,\phi_1)\in\mathcal H_c\).

\begin{Theorem}\label{thm3.2}
Under the assumptions of Theorem \ref{thm3.1}, the weak solution of equation \eqref{1.1} is unique. Moreover, let \(u\) and \(v\) be two weak solutions corresponding to initial histories \((\phi_{0,1},\phi_{1,1})\in\mathcal H_c\) and \((\phi_{0,2},\phi_{1,2})\in\mathcal H_c\), respectively. Then, for every \(T>\tau\), there exists a constant \(C_2>0\) such that
\begin{equation}\label{3.41}
\begin{aligned}
&\|u^t-v^t\|_{C_V}^2+\|(u_t)^t-(v_t)^t\|_{C_H}^2
\leq
C_2
\left(
\|\phi_{0,1}-\phi_{0,2}\|_{C_V}^2
+\|\phi_{1,1}-\phi_{1,2}\|_{C_H}^2
\right),
\qquad t\in[\tau,T],
\end{aligned}
\end{equation}
where \(R\) is the radius appearing in \eqref{1.10} and depends on the bounds of the two initial histories, while \(C_2=C_2(T,R)>0\) is independent of \(u-v\).
\end{Theorem}

{\bf Proof.}
Let \(w(t)=u(t)-v(t)\). Then \(w\) satisfies
\begin{equation}\label{3.42}
w_{tt}-\Delta w_t+\Delta^2w+\lambda w+f(u)-f(v)
=
G(t,u^t)-G(t,v^t),
\end{equation}
with \(w(t)=\phi_{0,1}(t-\tau)-\phi_{0,2}(t-\tau)\) and \(w_t(t)=\phi_{1,1}(t-\tau)-\phi_{1,2}(t-\tau)\) for \(t\in[\tau-r,\tau]\).

Taking the \(L^2(\Omega)\)-inner product of \eqref{3.42} with \(w_t(t)\), we obtain
\begin{equation}\label{3.43}
\begin{aligned}
&\frac{d}{dt}
\left(
\|w_t(t)\|^2+\|\Delta w(t)\|^2+\lambda\|w(t)\|^2
\right)
+2\|\nabla w_t(t)\|^2\\
&\leq
2(G(t,u^t)-G(t,v^t),w_t(t))
-2(f(u(t))-f(v(t)),w_t(t)).
\end{aligned}
\end{equation}

Applying the Poincar\'e and Young inequalities, together with assumption {\rm (G3)}, we obtain
\begin{equation}\label{3.44}
2|(G(t,u^t)-G(t,v^t),w_t(t))|
\leq
\frac12\|\nabla w_t(t)\|^2+C\|u^t-v^t\|^2.
\end{equation}

Since \(u\) and \(v\) are bounded in \(L^\infty(\tau,T;V)\), the local Lipschitz condition \eqref{1.10} gives
\begin{equation}\label{3.45}
2|(f(u(t))-f(v(t)),w_t(t))|
\leq
\frac12\|\nabla w_t(t)\|^2+C_R\|\Delta w(t)\|^2.
\end{equation}

Combining \eqref{3.43}--\eqref{3.45}, we get
\begin{equation}\label{3.46}
\begin{aligned}
&\frac{d}{dt}
\left(
\|w_t(t)\|^2+\|\Delta w(t)\|^2+\lambda\|w(t)\|^2
\right)
+\|\nabla w_t(t)\|^2\\
&\leq
C_R\left(\|w_t(t)\|^2+\|\Delta w(t)\|^2\right)
+C\|u^t-v^t\|^2.
\end{aligned}
\end{equation}

Let
\begin{equation}\label{3.47}
Y(t)=\sup_{\rho\in[\tau,t]}
\left(
\|w_t(\rho)\|^2+\|\Delta w(\rho)\|^2+\lambda\|w(\rho)\|^2
\right).
\end{equation}

For \(s\in[\tau,t]\), the history norm satisfies
\begin{equation}\label{3.48}
\|u^s-v^s\|_{C_H}^2
\leq
C\|\phi_{0,1}-\phi_{0,2}\|_{C_H}^2+C Y(s).
\end{equation}

Integrating \eqref{3.46} over \([\tau,t]\), using \eqref{3.48}, and taking the supremum over \([\tau,t]\), we obtain
\begin{equation}\label{3.49}
Y(t)
\leq
C\left(
\|\phi_{0,1}-\phi_{0,2}\|_{C_V}^2
+\|\phi_{1,1}-\phi_{1,2}\|_{C_H}^2
\right)
+
C_{T,R}\int_{\tau}^{t}Y(s)\,ds.
\end{equation}

By the Gronwall inequality, we conclude
\begin{equation}\label{3.50}
Y(t)
\leq
C_{T,R}
\left(
\|\phi_{0,1}-\phi_{0,2}\|_{C_V}^2
+\|\phi_{1,1}-\phi_{1,2}\|_{C_H}^2
\right),
\qquad t\in[\tau,T].
\end{equation}

Taking the supremum over \(s\in[t-r,t]\) and using the prescribed initial histories on \([\tau-r,\tau]\), we obtain \eqref{3.41}, which establishes the continuous dependence of weak solutions on the initial histories. If the two initial histories coincide, the right-hand side of \eqref{3.41} vanishes, and hence \(u=v\) on \([\tau-r,T]\), proving uniqueness. This proves the uniqueness and continuous dependence. \(\hfill\Box\)

By Theorems \ref{thm3.1} and \ref{thm3.2}, for every \(t\geq\tau\), we define
\begin{equation}\label{3.51}
U(t,\tau):\mathcal H_c\to\mathcal H_c,
\qquad
U(t,\tau)(\phi_0,\phi_1)=\bigl(u^t,(u_t)^t\bigr),
\end{equation}
where \(u\) is the unique weak solution of equation \eqref{1.1} with initial history \((\phi_0,\phi_1)\in\mathcal H_c\).

By Theorem \ref{thm3.1}, \(\bigl(u^t,(u_t)^t\bigr)\in\mathcal H_c\) for every \(t\geq\tau\), and therefore \(U(t,\tau)\) is well-defined on \(\mathcal H_c\). The uniqueness of solutions implies
\begin{equation}\label{3.52}
U(\tau,\tau)=I,\qquad
U(t,\tau)=U(t,s)U(s,\tau),
\qquad
\tau\leq s\leq t.
\end{equation}

Moreover, the continuous dependence estimate \eqref{3.41} shows that \(U(t,\tau)\) is continuous on bounded subsets of \(\mathcal H_c\). Hence \(\{U(t,\tau)\}_{t\geq\tau}\) is a continuous process on \(\mathcal H_c\).
\(\hfill\Box\)
\section{Pullback Attractors}
\setcounter{equation}{0}

In this section, we prove the existence of a pullback \(\mathcal D_\alpha\)-attractor in the compatible phase space \(\mathcal H_c\) for the process \(\{U(t,\tau)\}_{t\geq\tau}\) generated by equation \eqref{1.1}.

\subsection{Pullback \(\mathcal D_\alpha\)-absorbing Sets}

We first establish a pullback absorbing estimate for the process \(\{U(t,\tau)\}_{t\geq\tau}\). The main point is to construct a modified energy functional which is compatible with the strong damping term \(-\Delta u_t\) and the delay term \(G(t,u^t)\).

\begin{Lemma}\label{lem4.1}
Fix \(\tau\in\mathbb R\) and \(T>\tau\) and consider an initial history \((\phi_0,\phi_1)\in\mathcal H_c\) and the weak solution \(u\) of equation \eqref{1.1} satisfies, for all \(t\) with \(t-r\geq\tau\),
\begin{equation}\label{4.1}
\begin{aligned}
\|u^t\|_{C_V}^{2}+\|(u_t)^t\|_{C_H}^{2}
&\leq
C e^{-\alpha(t-r-\tau)}
\left(1+\|\phi_0\|_{C_V}^{2}+\|\phi_1\|_{C_H}^{2}\right)\\
&+C e^{-\alpha(t-r)}
\int_{\tau}^{t}e^{\alpha s}\|h(s)\|^{2}\,ds+C,
\end{aligned}
\end{equation}
where \(C>0\) is independent of \(t\), \(\tau\) and the initial data.
\end{Lemma}

{\bf Proof.}
Multiplying equation \eqref{1.1} by \(u_t\) and integrating over \(\Omega\), we obtain
\begin{equation}\label{4.2}
\frac{d}{dt}
\left(
\|u_t\|^{2}
+\|\Delta u\|^{2}
+\lambda\|u\|^{2}
+2\int_{\Omega}F(u)\,dx
\right)
+2\|\nabla u_t\|^{2}
=
2(G(t,u^t),u_t)+2(h(t),u_t).
\end{equation}

Multiplying equation \eqref{1.1} by \(u\) and integrating over \(\Omega\), we get
\begin{equation}\label{4.3}
\begin{aligned}
\frac{d}{dt}
\left(
2(u_t,u)+\|\nabla u\|^{2}
\right)
+2\|\Delta u\|^{2}
+2\lambda\|u\|^{2}
-2\|u_t\|^{2}
+2(f(u),u)
=
2(G(t,u^t),u)+2(h(t),u).
\end{aligned}
\end{equation}

Let \(0<\varepsilon<1\) be sufficiently small and define
\begin{equation}\label{4.4}
\begin{aligned}
\mathcal L(t)
&=
\|u_t(t)\|^{2}
+\|\Delta u(t)\|^{2}
+\lambda\|u(t)\|^{2}
+2\int_{\Omega}F(u(t))\,dx   \\
&\quad
+\varepsilon\left(2(u_t(t),u(t))+\|\nabla u(t)\|^{2}\right)
+K_0,
\end{aligned}
\end{equation}
where \(K_0>0\) is chosen large enough.

By \eqref{1.7}, the Poincar\'e and Young inequalities, there exist constants \(C_5,C_6>0\), independent of \(t\), such that
\begin{equation}\label{4.5}
C_5\left(\|u_t(t)\|^{2}+\|\Delta u(t)\|^{2}\right)
\leq
\mathcal L(t)
\leq
C_6\left(
1+\|u_t(t)\|^{2}
+\|\Delta u(t)\|^{2}
+\int_{\Omega}F(u(t))\,dx
\right).
\end{equation}

Using \eqref{4.2}--\eqref{4.5}, we conclude
\begin{equation}\label{4.6}
\begin{aligned}
\frac{d}{dt}\mathcal L(t)
&+2\|\nabla u_t\|^{2}
+2\varepsilon\|\Delta u\|^{2}
+2\varepsilon\lambda\|u\|^{2}
-2\varepsilon\|u_t\|^{2}
+2\varepsilon(f(u),u)      \\
&\leq
2(G(t,u^t),u_t)+2(h(t),u_t)
+2\varepsilon(G(t,u^t),u)+2\varepsilon(h(t),u).
\end{aligned}
\end{equation}

From \eqref{1.7}, together with the Poincar\'e and Young inequalities, and choosing \(\varepsilon>0\) and \(\delta_0>0\) sufficiently small, we obtain there exists a constant \(\beta>0\) such that
\begin{equation}\label{4.7}
\frac{d}{dt}\mathcal L(t)+\beta\mathcal L(t)
\leq
C\|G(t,u^t)\|^{2}
+C\|h(t)\|^{2}
+C.
\end{equation}

Multiplying \eqref{4.7} by \(e^{\beta t}\) and integrating the resulting inequality over \([\tau,t]\), we derive
\begin{equation}\label{4.8}
\begin{aligned}
\mathcal L(t)e^{\beta t}
&\leq
\mathcal L(\tau)e^{\beta\tau}
+C\int_{\tau}^{t}e^{\beta s}\|G(s,u^s)\|^{2}\,ds   \\
&+C\int_{\tau}^{t}e^{\beta s}\|h(s)\|^{2}\,ds
+C\int_{\tau}^{t}e^{\beta s}\,ds.
\end{aligned}
\end{equation}

By assumptions {\rm (G2)} and {\rm (G4)}, we obtain
\begin{equation}\label{4.9}
\int_{\tau}^{t}e^{\beta s}\|G(s,u^s)\|^{2}\,ds
\leq
C_G^{2}\int_{\tau-r}^{t}e^{\beta s}\|u(s)\|^{2}\,ds.
\end{equation}

Combining \eqref{4.9} with \(u(s)=\phi_0(s-\tau)\) for \(s\in[\tau-r,\tau]\) and \(\|u(s)\|^2\le C\mathcal L(s)\) for \(s\ge\tau\), we conclude
\begin{equation}\label{4.10}
\begin{aligned}
\mathcal L(t)e^{\beta t}
&\leq
C e^{\beta\tau}
\left(1+\|\phi_0\|_{C_V}^{2}+\|\phi_1\|_{C_H}^{2}\right)
+\delta_G \int_{\tau}^{t}e^{\beta s}\mathcal L(s)\,ds     \\
&+C\int_{\tau}^{t}e^{\beta s}\|h(s)\|^{2}\,ds
+C\int_{\tau}^{t}e^{\beta s}\,ds.
\end{aligned}
\end{equation}
Here \(\delta_G>0\) is the constant arising from the delay estimate in {\rm (G4)}. It depends only on \(C_1\), the domain-related constants and the equivalence constants in \eqref{4.5}, but is independent of \(t\), \(\tau\) and the particular solution.

Applying the Gronwall inequality to \eqref{4.10} and using
\begin{equation}\label{4.11}
\alpha=\beta-\delta_G>0,
\end{equation}
we derive
\begin{equation}\label{4.12}
\mathcal L(t)
\leq
C e^{-\alpha(t-\tau)}
\left(1+\|\phi_0\|_{C_V}^{2}+\|\phi_1\|_{C_H}^{2}\right)
+C e^{-\alpha t}\int_{\tau}^{t}e^{\alpha s}\|h(s)\|^{2}\,ds
+C.
\end{equation}

Combining \eqref{4.5} with \eqref{4.12}, replacing \(t\) by \(t+\theta\), and taking the supremum over \(\theta\in[-r,0]\), we obtain \eqref{4.1}. \(\hfill\Box\)

\begin{Remark}\label{rem4.2}
The constant \(\delta_G\) in \eqref{4.10} is generated by the estimate of the delay term \(G(t,u^t)\) through condition {\rm (G4)}. It depends on the delay constant \(C_7\), the embedding and Poincar\'e constants associated with \(\Omega\), and the constants involved in the energy equivalence \eqref{4.5}. Thus, the condition \(\alpha=\beta-\delta_G>0\) ensures that the system dissipation dominates the delay-induced growth. This condition \(\alpha=\beta-\delta_G>0\) yields the positive decay rate required for the exponential pullback estimate and the construction of a pullback absorbing family.

\end{Remark}

\begin{Corollary}\label{cor4.3}
Under the assumptions of Lemma \ref{lem4.1}, if moreover\eqref{1.13} and \eqref{4.11} hold, then the family \(\widehat D_0=\{D_0(t):t\in\mathbb R\}\), defined by \(D_0(t)=B_{\mathcal H_c}(0,\rho(t))\), where
\begin{equation}\label{4.13}
\rho^{2}(t)
=
C\left(
1+
e^{-\alpha(t-r)}
\int_{-\infty}^{t}e^{\alpha s}\|h(s)\|^{2}\,ds
\right),
\end{equation}
is a pullback \(\mathcal D_\alpha\)-absorbing family for the process \(\{U(t,\tau)\}_{t\geq\tau}\) on \(\mathcal H_c\). Moreover, \(\widehat D_0\in\mathcal D_\alpha\).
\end{Corollary}

{\bf Proof.}
Let \(t\in\mathbb R\) and \(\widehat D=\{D(t):t\in\mathbb R\}\in\mathcal D_\alpha\) be fixed. By \eqref{4.1}, for any \((\phi_0,\phi_1)\in D(\tau)\subset\mathcal H_c\) and \(t-r\geq\tau\), we conclude
\begin{equation}\label{4.14}
\begin{aligned}
\|U(t,\tau)(\phi_0,\phi_1)\|_{\mathcal H_c}^{2}
&\leq
C e^{-\alpha(t-r-\tau)}
\left(1+\|(\phi_0,\phi_1)\|_{\mathcal H_c}^{2}\right)   \\
&+C e^{-\alpha(t-r)}
\int_{\tau}^{t}e^{\alpha s}\|h(s)\|^{2}\,ds+C.
\end{aligned}
\end{equation}

Since \(\widehat D\in\mathcal D_\alpha\), the first term on the right-hand side of \eqref{4.14} tends to zero as \(\tau\to-\infty\). Hence there exists \(\tau_0=\tau_0(t,\widehat D)\leq t-r\) such that
\begin{equation}\label{4.16}
U(t,\tau)D(\tau)\subset D_0(t),
\qquad \tau\leq\tau_0.
\end{equation}
Thus \(\widehat D_0\) is pullback \(\mathcal D_\alpha\)-absorbing.

It remains to verify that \(\widehat D_0\in\mathcal D_\alpha\). By \eqref{1.13} and \eqref{4.14}, for \(t\leq0\),
\begin{equation}\label{4.17}
e^{\alpha t}\rho^{2}(t)
\leq
C e^{\alpha t}
+C e^{\alpha r}
\int_{-\infty}^{t}e^{\alpha s}\|h(s)\|^{2}\,ds
\to0
\quad \text{as } t\to-\infty.
\end{equation}

Therefore, \(\widehat D_0\in\mathcal D_\alpha\). \(\hfill\Box\)

\subsection{Pullback \(\mathcal D_\alpha\)-attractors}

We now prove the pullback \(\mathcal D_\alpha\)-asymptotic compactness of the process \(\{U(t,\tau)\}_{t\geq\tau}\) in \(\mathcal H_c\). The proof is based on the contractive function method.

\begin{Theorem}\label{thm4.4}
Under the assumptions of Lemma \ref{lem4.1}, if \eqref{1.13} and
\eqref{4.11} hold, the process \(\{U(t,\tau)\}_{t\geq\tau}\)
is pullback \(\mathcal D_\alpha\)-asymptotically compact in
\(\mathcal H_c\).
\end{Theorem}

{\bf Proof.}
Fix \(t\in\mathbb R\). By Corollary \ref{cor4.3}, it is sufficient to establish pullback asymptotic compactness on the absorbing family \(\widehat D_0=\{D_0(t):t\in\mathbb R\}\). Choose \(T_0>2r\) and set \(\tau=t-T_0\). For \(i=1,2\), let \((\phi_{0,i},\phi_{1,i})\in D_0(\tau)\), and denote by \(u_i\) the corresponding solution of \eqref{1.1}. Setting \(w=u_1-u_2\), \(w\) satisfies
\begin{equation}\label{4.18}
w_{tt}-\Delta w_t+\Delta^2w+\lambda w+f(u_1)-f(u_2)
=
G(t,u_1^t)-G(t,u_2^t).
\end{equation}

Define
\begin{equation}\label{4.19}
E_w(t)
=
\frac12\|w_t(t)\|^{2}
+\frac12\|\Delta w(t)\|^{2}
+\frac{\lambda}{2}\|w(t)\|^{2}.
\end{equation}

Taking the inner product of \eqref{4.18} with \(w_t\) in \(L^2(\Omega)\), we obtain
\begin{equation}\label{4.20}
\begin{aligned}
\frac{d}{dt}E_w(t)+\|\nabla w_t(t)\|^{2}
&=
(G(t,u_1^t)-G(t,u_2^t),w_t(t))  \\
&-(f(u_1(t))-f(u_2(t)),w_t(t)).
\end{aligned}
\end{equation}

For simplicity, set
\begin{equation}\label{4.21}
g_w(s)=G(s,u_1^s)-G(s,u_2^s),
\qquad
f_w(s)=f(u_1(s))-f(u_2(s)).
\end{equation}

Then \eqref{4.20} can be rewritten as
\begin{equation}\label{4.22}
\frac{d}{ds}E_w(s)+\|\nabla w_t(s)\|^2
=(g_w(s),w_t(s))-(f_w(s),w_t(s)).
\end{equation}

By the Poincar\'e and Young inequalities, we conclude
\begin{equation}\label{4.23}
\begin{aligned}
|(g_w(s),w_t(s))|
&\leq
\|g_w(s)\|\,\|w_t(s)\|        \\
&\leq
C\|g_w(s)\|\,\|\nabla w_t(s)\| \\
&\leq
\frac14\|\nabla w_t(s)\|^2+C\|g_w(s)\|^2 .
\end{aligned}
\end{equation}

Combining \eqref{4.22} with \eqref{4.23}, we get
\begin{equation}\label{4.24}
\frac{d}{ds}E_w(s)+\frac34\|\nabla w_t(s)\|^2
\leq
C\|g_w(s)\|^2+|(f_w(s),w_t(s))|.
\end{equation}

Multiplying \eqref{4.24} by \(e^{\alpha s}\), we infer that
\begin{equation}\label{4.25}
\begin{aligned}
\frac{d}{ds}\left(e^{\alpha s}E_w(s)\right)
+\frac34 e^{\alpha s}\|\nabla w_t(s)\|^2
&\leq
C e^{\alpha s}E_w(s)
+C e^{\alpha s}\|g_w(s)\|^2       \\
&+e^{\alpha s}|(f_w(s),w_t(s))|.
\end{aligned}
\end{equation}

To estimate \(E_w(s)\), we take \(w(s)\) as a test function in
\eqref{4.18}. Integrating by parts and using the boundary conditions,
the definition of \(E_w(s)\) and the Poincar\'e inequality, we obtain
\begin{equation}\label{4.26}
\begin{aligned}
E_w(s)
&\leq
C\|\nabla w_t(s)\|^2
-\frac{d}{ds}(w_t(s),w(s))
-\frac12\frac{d}{ds}\|\nabla w(s)\|^2 \\
&\quad
+C\bigl|(f_w(s),w(s))\bigr|
+C\bigl|(g_w(s),w(s))\bigr|.
\end{aligned}
\end{equation}

Combining \eqref{4.26} with the preceding energy inequality, integrating the resulting relation over \([s,q]\), where \(\tau\leq s\leq q\), and absorbing the term involving \(\|\nabla w_t\|^2\) into the dissipation generated by the strong damping, we derive
\begin{equation}\label{4.27}
\begin{aligned}
E_w(q)
&\leq
C e^{-\alpha(q-s)}E_w(s)
+C|(w_t(q),w(q))|
+C\|\nabla w(q)\|^2  \\
&+C e^{-\alpha q}\int_s^q e^{\alpha \xi}\|g_w(\xi)\|^2\,d\xi
+C e^{-\alpha q}\int_s^q e^{\alpha \xi}|(g_w(\xi),w(\xi))|\,d\xi  \\
&+C e^{-\alpha q}\int_s^q e^{\alpha \xi}\|w(\xi)\|^2\,d\xi
+C e^{-\alpha q}\int_s^q e^{\alpha \xi}|(f_w(\xi),w_t(\xi))|\,d\xi  \\
&+C e^{-\alpha q}\int_s^q e^{\alpha \xi}|(f_w(\xi),w(\xi))|\,d\xi .
\end{aligned}
\end{equation}

Indeed, by the Young inequality, we conclude
\begin{equation}\label{4.28}
|(g_w(\xi),w(\xi))|
\leq
C\|g_w(\xi)\|^2+C\|w(\xi)\|^2.
\end{equation}

Moreover, by {\rm (G4)}, we derive
\begin{equation}\label{4.29}
\int_s^q e^{\alpha \xi}\|g_w(\xi)\|^2\,d\xi
\leq
C_G^2\int_{s-r}^{q}e^{\alpha \xi}\|w(\xi)\|^2\,d\xi.
\end{equation}

Splitting \([s-r,q]\) into the initial-history part and the solution part, we get
\begin{equation}\label{4.30}
e^{-\alpha q}\int_s^q e^{\alpha \xi}\|g_w(\xi)\|^2\,d\xi
\leq
C e^{-\alpha(q-\tau)}
\|\phi_{0,1}-\phi_{0,2}\|_{C_H}^2
+
C\int_{\tau}^{q}\|w(\xi)\|^2\,d\xi .
\end{equation}

The same estimate holds for the term involving \((g_w,w)\). Therefore, combining \eqref{4.27}--\eqref{4.30}, we conclude
\begin{equation}\label{4.31}
\begin{aligned}
E_w(q)
&\leq
C e^{-\alpha(q-s)}E_w(s)
+C e^{-\alpha(q-\tau)}
\|\phi_{0,1}-\phi_{0,2}\|_{C_H}^2  \\
&+C|(w_t(q),w(q))|
+C\|\nabla w(q)\|^2
+C\int_\tau^q\|w(\xi)\|^2\,d\xi  \\
&+C e^{-\alpha q}\int_s^q e^{\alpha \xi}|(f_w(\xi),w_t(\xi))|\,d\xi  \\
&+C e^{-\alpha q}\int_s^q e^{\alpha \xi}|(f_w(\xi),w(\xi))|\,d\xi .
\end{aligned}
\end{equation}

Integrating \eqref{4.31} with respect to \(s\) over \([\tau,q]\) and dividing by \(q-\tau\), the integrals depending on \(s\) generate the double integral terms. Since \(f_w=f(u_1)-f(u_2)\), we finally obtain
\begin{equation}\label{4.32}
\begin{aligned}
E_w(q)
&\leq
\frac{C}{q-\tau}e^{-\alpha(q-\tau)}E_w(\tau)
+C e^{-\alpha(q-\tau)}
\|\phi_{0,1}-\phi_{0,2}\|_{C_H}^{2}  \\
&+C|(w_t(q),w(q))|
+C\|\nabla w(q)\|^{2}
+C\int_{\tau}^{q}\|w(s)\|^{2}\,ds  \\
&+C e^{-\alpha q}\int_{\tau}^{q}e^{\alpha s}
|(f(u_1(s))-f(u_2(s)),w_t(s))|\,ds  \\
&+C e^{-\alpha q}\int_{\tau}^{q}e^{\alpha s}
|(f(u_1(s))-f(u_2(s)),w(s))|\,ds  \\
&+C e^{-\alpha q}\int_{\tau}^{q}\int_{\sigma}^{q}e^{\alpha s}
|(f(u_1(s))-f(u_2(s)),w_t(s))|\,dsd\sigma  \\
&+C e^{-\alpha q}\int_{\tau}^{q}\int_{\sigma}^{q}e^{\alpha s}
|(f(u_1(s))-f(u_2(s)),w(s))|\,dsd\sigma .
\end{aligned}
\end{equation}

Since \(q\in[t-r,t]\) and \(\tau=t-T_0\), we have \(q-\tau\geq T_0-r\). Hence, for \(T_0>2r\) sufficiently large, the factor \(e^{-\alpha(q-\tau)}\) is uniformly small for all \(q\in[t-r,t]\). Moreover, the absorbing property of \(D_0(\tau)\) implies that \(E_w(\tau)\) and \(\|\phi_{0,1}-\phi_{0,2}\|_{C_H}^2\) are uniformly controlled for \((\phi_{0,i},\phi_{1,i})\in D_0(\tau)\). Therefore, the first two terms on the right-hand side of \eqref{4.32} can be made smaller than any prescribed \(\varepsilon>0\) by taking \(T_0\) large enough. Thus, it follows that
\vspace{-0.5em}
\begin{equation}\label{4.33}
\sup_{q\in[t-r,t]}E_w(q)\leq\varepsilon+\Psi_{t,T_0}\left((\phi_{0,1},\phi_{1,1}),(\phi_{0,2},\phi_{1,2})\right),
\end{equation}
\vspace{-0.8em}
where \(\Psi_{t,T_0}\) is defined by
\begingroup
\setlength{\abovedisplayskip}{2pt}
\setlength{\belowdisplayskip}{2pt}
\setlength{\abovedisplayshortskip}{2pt}
\setlength{\belowdisplayshortskip}{2pt}
\setlength{\jot}{1pt}
\begin{equation}\label{4.34}
\begin{aligned}
&\Psi_{t,T_0}\left((\phi_{0,1},\phi_{1,1}),
(\phi_{0,2},\phi_{1,2})\right)\\[-1pt]
&=C\sup_{q\in[t-r,t]}\Bigg[
|(w_t(q),w(q))|+\|\nabla w(q)\|^2
+\int_\tau^q\|w(s)\|^2\,ds\\[-1pt]
&+e^{-\alpha q}\int_\tau^q e^{\alpha s}
|(f(u_1(s))-f(u_2(s)),w_t(s))|\,ds\\[-1pt]
&+e^{-\alpha q}\int_\tau^q e^{\alpha s}
|(f(u_1(s))-f(u_2(s)),w(s))|\,ds\\[-1pt]
&+e^{-\alpha q}\int_\tau^q\int_\sigma^q e^{\alpha s}
|(f(u_1(s))-f(u_2(s)),w_t(s))|\,ds\,d\sigma\\[-1pt]
&+e^{-\alpha q}\int_\tau^q\int_\sigma^q e^{\alpha s}
|(f(u_1(s))-f(u_2(s)),w(s))|\,ds\,d\sigma
\Bigg].
\end{aligned}
\end{equation}
\endgroup

To show that
\(\Psi_{t,T_0}\in\operatorname{Contr}(D_0(\tau))\), consider any
sequence \(\{(\phi_{0,n},\phi_{1,n})\}_{n=1}^{\infty}\subset
D_0(\tau)\), and denote by \(u_n\) the corresponding solution.
Since \(D_0(\tau)\) is bounded in \(\mathcal H_c\), estimates
\eqref{3.12}, \eqref{3.15}, and \eqref{3.18} imply the existence of
a constant \(C>0\), independent of \(n\), such that
\begin{equation}\label{4.35}
\|u_n\|_{L^\infty(\tau,t;V)}
+\|(u_n)_t\|_{L^\infty(\tau,t;H)}
+\|(u_n)_t\|_{L^2(\tau,t;H_0^1(\Omega))}
+\|(u_n)_{tt}\|_{L^2(\tau,t;V')}
\leq C .
\end{equation}

We first explain the compactness of the displacement sequence. Since \(V\hookrightarrow\hookrightarrow H_0^1(\Omega)\hookrightarrow H\), and since \(\{u_n\}\) is bounded in \(L^\infty(\tau,t;V)\) while \(\{\partial_tu_n\}\) is bounded in \(L^2(\tau,t;H)\), the Aubin--Lions--Simon compactness theorem implies that there exist a subsequence, still denoted by \(\{u_n\}\), and a function \(u\) such that
\begin{equation}\label{4.36}
u_n\to u
\quad \text{strongly in } C([t-r,t];H_0^1(\Omega)).
\end{equation}
This compactness is obtained by applying the Simon version of the Aubin--Lions theorem with \(X_0=V\), \(X=H_0^1(\Omega)\) and \(X_1=H\). The embedding \(V\hookrightarrow\hookrightarrow H_0^1(\Omega)\) is compact, while \(H_0^1(\Omega)\hookrightarrow H\) is continuous.

We next explain the compactness of the velocity sequence. Since \(\{(u_n)_t\}\) is bounded in \(L^2(\tau,t;H_0^1(\Omega))\) and \(\{\partial_t (u_n)_t\}=\{(u_n)_{tt}\}\) is bounded in \(L^2(\tau,t;V')\), and since \(H_0^1(\Omega)\hookrightarrow\hookrightarrow H\hookrightarrow V'\), the Aubin--Lions compactness theorem ensures the existence of a subsequence, still denoted by \(\{(u_n)_t\}\), such that
\begin{equation}\label{4.37}
(u_n)_t\to u_t
\quad \text{strongly in } L^2(\tau,t;H).
\end{equation}

Consequently, setting \(w_{nm}=u_n-u_m\), we derive
\begin{equation}\label{4.38}
\lim_{n\to\infty}\lim_{m\to\infty}
\sup_{q\in[t-r,t]}\|w_{nm}(q)\|_{H_0^1(\Omega)}^{2}=0,
\end{equation}
and
\begin{equation}\label{4.39}
\lim_{n\to\infty}\lim_{m\to\infty}
\int_{\tau}^{t}\|(u_n)_t(s)-(u_m)_t(s)\|^{2}\,ds=0.
\end{equation}

Moreover, \eqref{4.36} also implies
\begin{equation}\label{4.40}
\lim_{n\to\infty}\lim_{m\to\infty}
\int_{\tau}^{t}\|u_n(s)-u_m(s)\|^{2}\,ds=0.
\end{equation}

We now verify that each term in \(\Psi_{t,T_0}\) has double limit zero along the above subsequence. First, by the uniform boundedness of \(\{(u_n)_t\}\) in \(L^\infty(\tau,t;H)\) and \eqref{4.40}, we obtain
\begin{equation}\label{4.41}
\begin{aligned}
&\sup_{q\in[t-r,t]}
\left|((u_n)_t(q)-(u_m)_t(q),w_{nm}(q))\right|       \\
&\qquad
\leq
C\sup_{q\in[t-r,t]}\|w_{nm}(q)\|
\to0
\quad \text{as } n\to\infty,\ m\to\infty
\end{aligned}
\end{equation}
in the double-limit sense. Therefore, we derive
\begin{equation}\label{4.42}
\lim_{n\to\infty}\lim_{m\to\infty}
\sup_{q\in[t-r,t]}
\left|((u_n)_t(q)-(u_m)_t(q),w_{nm}(q))\right|=0.
\end{equation}

Second, from the strong convergence in \(C([t-r,t];H_0^1(\Omega))\), we directly conculde
\begin{equation}\label{4.43}
\lim_{n\to\infty}\lim_{m\to\infty}
\sup_{q\in[t-r,t]}\|\nabla w_{nm}(q)\|^{2}=0.
\end{equation}

Third, by \eqref{4.40}, we obtain
\begin{equation}\label{4.44}
\lim_{n\to\infty}\lim_{m\to\infty}
\sup_{q\in[t-r,t]}
\int_{\tau}^{q}\|w_{nm}(s)\|^{2}\,ds=0.
\end{equation}

It remains to estimate the nonlinear terms in \(\Psi_{t,T_0}\). By the growth condition \eqref{1.9} and the uniform boundedness of \(\{u_n\}\) in \(L^\infty(\tau,t;V)\), the sequence \(\{f(u_n)\}\) is bounded in \(L^2(\tau,t;H)\). Therefore, there exists \(C>0\), independent of \(n,m\), such that
\begin{equation}\label{4.45}
\|f(u_n)-f(u_m)\|_{L^2(\tau,t;H)}\leq C.
\end{equation}

Combining \eqref{4.39} with the Cauchy--Schwarz inequality, we derive
\begin{equation}\label{4.46}
\begin{aligned}
&\int_{\tau}^{t}e^{\alpha s}
\left|
(f(u_n(s))-f(u_m(s)),(u_n)_t(s)-(u_m)_t(s))
\right|\,ds       \\
&\qquad
\leq
C
\|f(u_n)-f(u_m)\|_{L^2(\tau,t;H)}
\|(u_n)_t-(u_m)_t\|_{L^2(\tau,t;H)}
\to0.
\end{aligned}
\end{equation}

Hence, we conclude
\begin{equation}\label{4.47}
\lim_{n\to\infty}\lim_{m\to\infty}
\int_{\tau}^{t}e^{\alpha s}
\left|
(f(u_n(s))-f(u_m(s)),(u_n)_t(s)-(u_m)_t(s))
\right|\,ds=0.
\end{equation}

Similarly, by \eqref{4.40} and the Cauchy--Schwarz inequality, we obtain
\begin{equation}\label{4.48}
\begin{aligned}
&\int_{\tau}^{t}e^{\alpha s}
\left|
\bigl(f(u_n(s))-f(u_m(s)),u_n(s)-u_m(s)\bigr)
\right|\,ds\\
&\qquad\leq
C\|f(u_n)-f(u_m)\|_{L^2(\tau,t;H)}
\|u_n-u_m\|_{L^2(\tau,t;H)}
\longrightarrow0
\quad\text{as }n,m\to\infty,
\end{aligned}
\end{equation}
where \(C>0\) is independent of \(n\) and \(m\).

Consequently,
\begin{equation}\label{4.49}
\lim_{n,m\to\infty}
\int_{\tau}^{t}e^{\alpha s}
\left|
\bigl(f(u_n(s))-f(u_m(s)),u_n(s)-u_m(s)\bigr)
\right|\,ds=0.
\end{equation}

For the last two double-integral terms in \eqref{4.34}, the Fubini theorem together with \(q\leq t\) gives the corresponding single-integral estimates. Therefore, by \eqref{4.47} and \eqref{4.49}, we obtain
\begin{equation}\label{4.50}
\begin{aligned}
&e^{-\alpha q}
\int_{\tau}^{q}\int_{\sigma}^{q}e^{\alpha s}
\left|
(f(u_n(s))-f(u_m(s)),(u_n)_t(s)-(u_m)_t(s))
\right|\,dsd\sigma       \\
&\qquad
\leq
C_{t,\tau}
\int_{\tau}^{t}
\left|
(f(u_n(s))-f(u_m(s)),(u_n)_t(s)-(u_m)_t(s))
\right|\,ds,
\end{aligned}
\end{equation}
where \(C_{t,\tau}>0\) is independent of \(n,m\) and \(q\in[t-r,t]\). Combining \eqref{4.47} with \eqref{4.50}, we get
\begin{equation}\label{4.51}
\lim_{n\to\infty}\lim_{m\to\infty}
\sup_{q\in[t-r,t]}
e^{-\alpha q}
\int_{\tau}^{q}\int_{\sigma}^{q}e^{\alpha s}
\left|
(f(u_n(s))-f(u_m(s)),(u_n)_t(s)-(u_m)_t(s))
\right|\,dsd\sigma=0.
\end{equation}

Similarly, by \eqref{4.49}, we derive
\begin{equation}\label{4.52}
\lim_{n\to\infty}\lim_{m\to\infty}
\sup_{q\in[t-r,t]}
e^{-\alpha q}
\int_{\tau}^{q}\int_{\sigma}^{q}e^{\alpha s}
\left|
(f(u_n(s))-f(u_m(s)),u_n(s)-u_m(s))
\right|\,dsd\sigma=0.
\end{equation}

It follows from \eqref{4.42}--\eqref{4.44} and \eqref{4.47}--\eqref{4.52} that
\begin{equation}\label{4.53}
\lim_{n\to\infty}\lim_{m\to\infty}
\Psi_{t,T_0}\left((\phi_{0,n},\phi_{1,n}),(\phi_{0,m},\phi_{1,m})\right)=0.
\end{equation}
Hence, \(\Psi_{t,T_0}\) is a contractive function on \(D_0(\tau)\times D_0(\tau)\).

Finally, since $E_w(q)$ is equivalent to $\|w(q)\|_V^2+\|w_t(q)\|^2$ and $q\in[t-r,t]$, taking the supremum over the history interval yields the control of the $C_V\times C_H$-norm. Therefore, by \eqref{4.34}, we obtain
\begin{equation}\label{4.54}
\|U(t,\tau)(\phi_{0,1},\phi_{1,1})
-U(t,\tau)(\phi_{0,2},\phi_{1,2})\|_{\mathcal H_c}^{2}
\leq
\varepsilon+
\Psi_{t,T_0}\left((\phi_{0,1},\phi_{1,1}),(\phi_{0,2},\phi_{1,2})\right).
\end{equation}

By Corollary \ref{cor4.3}, $\widehat D_0$ is a pullback $\mathcal D_\alpha$-absorbing family. Moreover, \eqref{4.53} and \eqref{4.54} show that $\Psi_{t,T_0}\in\operatorname{Contr}(D_0(t-T_0))$ and that the contractive estimate required in Theorem \ref{thm2.5} holds. Therefore, the process $\{U(t,\tau)\}_{t\geq\tau}$ is pullback $\mathcal D_\alpha$-asymptotically compact in $\mathcal H_c$. \hfill$\Box$

Combining Corollary \ref{cor4.3} with Theorem \ref{thm4.4}, we obtain the following main result.

\begin{Theorem}\label{thm4.5}
Under the assumptions \eqref{1.13} and \eqref{4.11}, the process
\(\{U(t,\tau)\}_{t\geq\tau}\) generated by equation \eqref{1.1}
possesses a unique pullback \(\mathcal D_\alpha\)-attractor
\(\widehat{\mathcal A}=\{\mathcal A(t):t\in\mathbb R\}\) in
\(\mathcal H_c\). Moreover, \(\widehat{\mathcal A}\) pullback attracts
every \(\widehat D\in\mathcal D_\alpha\) with respect to the
\(\mathcal H_c\)-norm.
\end{Theorem}

{\bf Proof.}
By Corollary \ref{cor4.3}, the process \(\{U(t,\tau)\}_{t\geq\tau}\) has a pullback \(\mathcal D_\alpha\)-absorbing family \(\widehat D_0\in\mathcal D_\alpha\) in \(\mathcal H_c\). By Theorem \ref{thm4.4}, this process is pullback \(\mathcal D_\alpha\)-asymptotically compact in \(\mathcal H_c\). Therefore, by Theorem \ref{thm2.6}, there exists a unique pullback \(\mathcal D_\alpha\)-attractor \(\widehat{\mathcal A}\) in \(\mathcal H_c\). \(\hfill\Box\)

$\mathbf{Funding}$


This work was supported by the National Natural Science Foundation of China (Grant Nos. 12462005 and 62063025), the Natural Science Foundation of Inner Mongolia Autonomous Region of China (Grant Nos. 2024MS06028, 2026QC0419, and 2026MS0478), and the Keju Plan of Inner Mongolia University of Science and Technology (Grant Nos. KJJH2024985 and KJJH2025991).
$\mathbf{Conflict\,\,of\,\,interest\,\,statement}$

The authors have no conflict of interest.

$\mathbf{Acknowledgment}$

The authors sincerely thank the editors for their careful handling of the manuscript and the anonymous reviewers for their valuable and constructive comments, which helped improve its quality.

\end{document}